\documentclass[12pt]{amsart}
\usepackage[a4paper]{geometry}
\usepackage{xcolor}
\usepackage[all]{xy}
\usepackage{hyperref}
\usepackage{array}   
    \newcolumntype{L}{>{$}l<{$}} 
\usepackage{amssymb}

\newtheorem{theorem}{Theorem}[section]

\newtheorem{cor}[theorem]{Corollary}
\newtheorem{lemma}[theorem]{Lemma}
\theoremstyle{definition}

\theoremstyle{remark}

\title[Polynomials with random multiplicative coefficients]{Irreducibility of polynomials with random  multiplicative coefficients revisited}
\author{Oleksiy Klurman}
\address{School of Mathematics, University of Bristol, Bristol, United Kingdom}
\email{ oleksiy.klurman@bristol.ac.uk}

\author{Vlad Matei}
\address{Institute of Mathematics of the Romanian Academy\\  Bucharest\\Romania}
\email{vmatei@imar.ro}

\date{\today}

\begin{document}

\begin{abstract}

We show that for a random polynomial
\[
F(X) = \sum_{n=1}^{N} f(n) X^{n-1},
\]
where $f(n)$ is a random completely multiplicative function taking values in $\{\pm 1\}$, one has
\[
\limsup_{N \to \infty} \mathbb{P}\big[F(X) \text{ is irreducible}\big] = 1.
\]
\end{abstract}

\maketitle

\section{Introduction}
The study of the irreducibility of polynomials with integer coefficients has a long and rich history. One may ask this question for a specific polynomial or study various families of polynomials. The first model is the ``large box model,'' where we fix the degree of a polynomial and let the coefficients be drawn at random from the box $[-H, H]$ as $H \to \infty$. The papers \cite{bhargava2021galois}, \cite{chow2021towards}, and the references therein describe some recent spectacular developments in this subject.  

Another setting that has led to fruitful developments considers polynomials whose bounded coefficients are sampled independently from a fixed distribution law, while the degree of the polynomials in the family grows. The works \cite{bary2020irreducibility}, \cite{bary2020irreducible}, and \cite{Bazin2025} establish various exciting instances of this universality phenomenon for irreducibility---either when the coefficients are drawn from a sufficiently large finite set of integers or when they are near-uniformly distributed modulo enough primes. Conditionally on the Extended Riemann Hypothesis (ERH) over number fields, Breuillard and Varjú \cite{breuillard2019irreducibility} proved that polynomials with coefficients that are independent $\{0,1\}$-valued random variables, each taking the value $1$ with probability $1/2$, are almost surely irreducible. This conditionally resolved an old conjecture of Odlyzko and Poonen (\cite{odlyzko1993}). Their method also yields sharp estimates for the probability that such a random polynomial is reducible.  

Our work is motivated by the general question of irreducibility when the coefficients exhibit some form of dependence. There are several natural frameworks in which such dependencies arise, particularly in the study of characteristic polynomials of various ensembles of random matrices (see \cite{ferber2022random}, \cite{eberhard2022characteristic}, \cite{barrytrid} for recent results in this direction).  

Here we confine ourselves to a more arithmetic framework, where the dependencies are related to multiplicative functions. The most studied example of such families of polynomials is the well-known Fekete polynomials of degree $\vert D \vert - 1$, defined by  
$$F_D(z) := \sum_{n=1}^{\vert D \vert - 1} \chi_D(n) z^n,$$  
where $D$ is any fundamental discriminant and $\chi_D$ is the corresponding quadratic character. The properties of the real and complex zeros of such polynomials, as well as conjectures concerning their irreducibility, have been the subject of extensive recent research \cite{KLM}, \cite{BKS}, \cite{CGSP}, \cite{Minac}, \cite{Baker-Mont}, \cite{KLM1}.  

To facilitate our discussion, let $\{f(p)\}_p$ be a sequence of independent Bernoulli $\pm 1$ random variables indexed by the primes, and define  
$$f(n) = \prod_{p^k \| n} f^k(p)$$  
to be a random completely multiplicative function. The statistical behavior of the closely related Rademacher random multiplicative model has seen tremendous recent progress, most notably due to the work of Harper \cite{harper} on Helson's conjecture.  
Let $f : \mathbb{N} \to \{-1,1\}$ be a random completely multiplicative function, and consider the family of polynomials  
$$P_{f,N}(X) = \sum_{n=1}^{N} f(n) X^{n-1}.$$
Analytic properties of such objects have been extensively studied in the recent works \cite{hardy2024} and \cite{Benatar2022}.\\
Very recently, Varjú and Xu \cite{varju}, conditionally on the Extended Riemann Hypothesis, proved that such polynomials are irreducible with probability $1 - o_{N \to \infty}(1)$. We instead unconditionally establish this result for a subsequence of degrees, inspired by the work of Bary-Soroker, Hokken, Kozma, and Poonen \cite{barryspecial}, who addressed the analogous question for independent coefficients. Our main result is the following.

\begin{theorem} \label{mainthm} Let $N=2^k$. Then

$$\lim_{k\rightarrow \infty}\mathbb{P} [P_{f,N}(X) \hspace{2mm} \text{is irreducible}]=1.$$
\end{theorem}

In the last section we briefly discuss application of this result to the deterministic family of Turyn polynomials.

\subsection*{Acknowledgments}
The second author is supported by the project “Group schemes, root systems, and related representations” founded by the European Union - NextGenerationEU through Romania’s National Recovery and Resilience Plan (PNRR) call no. PNRR-III-C9-2023- I8, Project CF159/31.07.2023, and coordinated by the Ministry of Research, Innovation and Digitalization (MCID) of Romania.
THe authors would like to thank Lior Bary-Soroker and David Hoken for the comments on a preliminary version of this manuscript. 
The authors would like to thank the Oberwolfach Mathematics Institute for hosting them and the organizers of the workshops (K. Matomäki, J. Maynard , K. Soundararajan, T.D. Wooley and L. Bary-Soroker, A. Ostafe, and P. Sarnak) that fascilitated this collaboration. 

\section{Proof of Main Theorem}
In what follows, we shall always assume that all constants are sufficiently large unless stated otherwise.
\begin{lemma}\label{coeffdistr}
Consider $\tilde{P}_{f,N}(X)=P_{f,N}(X+1)=\sum_{j\le N}a_jx^j$ where $N=2^k$ and $j\geq 1$ and fix large $A\ll \log\log N.$ Then
$$\mathbb{P}(a_i\equiv 0 \pmod{4} \hspace{3mm} \forall i \in\mathcal{A}) \leq \exp(-cA),$$
where $\mathcal{A}=\{ 2^{A+1}+2^{i}+1,|\hspace{3mm} 1\leq i \leq A\}.$
   
\end{lemma}

\begin{proof} 

First, we note that
$$\tilde{P}_{f,N}(X)=\sum_{l=0}^{N-1} X^l \left(\sum_{l+1\leq n \leq N} f(n) \binom{n-1}{l}\right), $$
and introduce
$$\displaystyle g_l=\sum_{l+1\leq n \leq N} f(n) \binom{n-1}{l}.$$
Denoting by $\tilde{f}(n)=\cfrac{f(n)+1}{2}$ we rewrite  
$$g_l=2\sum_{l+1\leq n \leq N} \tilde{f}(n) \binom{n-1}{l}-\binom{N}{l+1}.$$
 For $N=2^k,$ we have that $\dbinom{2^k}{j}$ is even for $1\leq j\leq 2^k-1$. Thus we have that $g_l$ is even for $0\leq l\leq N-1 $ and, moreover, its residue modulo $4$ is determined by the parity of
 $$\sum_{l+1\leq n \leq N} \tilde{f}(n) \binom{n-1}{l}.$$
 For each $a\in \mathcal{A}$ we can pick a prime $p_a$ in $\left[ \cfrac{N}{2}, N\right]$ such that $p_a\equiv a \pmod{2^{A+1}}$  by the uniform version of the prime number theorem for arithmetic progressions (Theorem 1.1 in \cite{Koukoul_PNT}). Using Lucas's theorem (see Theorem 2.5.14 in \cite{moll}), we deduce that  $\dbinom{p_a-1}{a}$ is odd and $\dbinom{p_a-1}{b}$ is even for any $a\neq b \in\mathcal{A}$.
All that remains to note is that for coefficients $g_a$ with $a\in \mathcal{A}$   the values of $\tilde{f}(p_a)$ behave like independent random variables taking values $0,1$ and, moreover, $g_a\equiv \tilde{f}(p_a)+h_a\pmod{2}$ where the terms $h_a$ do not involve  any other term  $\tilde{f}(p_b)$ with $b\neq a$. Thus 
$$\mathbb{P}(a_i\equiv 0 \pmod{4} \hspace{3mm} \forall i \in\mathcal{A}) =\cfrac{1}{2^{A}}$$ and the claim is proved. 
 
\end{proof}

\begin{lemma}\label{iredfactor} With probability $(1- \exp(-cA))$ we have that $P_{f,N}(X)$ has an irreducible factor of degree $\geq N- 2^A.$

\end{lemma}

\begin{proof}

Consider the Newton polygon associated to the polynomial $\tilde{P}_{f,N}\left(X\right)$ over $\mathbb{Q}_2$. With probability $(1- \exp(-cA))$ there exists a coefficient with $1\leq j\leq 2^A$ such that $g_j\equiv 2 \pmod{4}$. Thus, the Newton polygon would have a segment of  height $1$ and length at least  $N-j$, and hence an irreducible factor of degree at least $N-j$ over $\mathbb{Q}_2$ (see, for example, \cite{gouvea2020}, section 7.4). Consequently $P_{f,N}(X)$ must have an irreducible factor of degree at least $N-j$ over $\mathbb{Q}$ with high probability claimed above. 
\end{proof}

To finish the proof, we need to bound the probability that our random polynomial has a low degree factor. To this end, we work with an ``independent" part of our polynomial given by   
$$\displaystyle \sum_{N/2\leq p\leq N,\hspace{1mm} p\hspace{1mm} \text{prime}} f(p)X^{p-1}.$$

\begin{lemma}\label{rootbound} We have that for $B\leq \log\log  N$
$$\mathbb{P}[P_{f,N}(X) \hspace{3mm} \text{ has a factor of   degree $\leq B$} \hspace{1mm}]=O(N^{-1/2+\varepsilon}).$$

\end{lemma}

\begin{proof}
We will make use of Theorem $1.7$ in \cite{rourke2019}. In their notations, we take $\Omega=\{z\in\mathbb{C}: |z|\leq 2\} $, with $M=2.$ More precisely, if we show that we can find $p\in [0,1]$ such that 
$$\sup_{z\in \Omega} \mathbb{P}(P_{f,N}(z)=0)\leq p,$$
then the probability that our polynomial $P_{f,N}(X)$ has a factor of degree $\leq B$ is bounded by 
$$p(2e)^{B^2},$$
since any root of $P_{f,N}(X),$ say $z_0$ satisfies the inequality $|z_0|< 2.$
We now focus on bounding probability $\mathbb{P}(P_{f,N}(z)=0)$ for every $z\in \Omega$ and distinguish between two cases depending on the size of $|z|-1.$ Since the product of the roots is equal to $\pm 1,$ we can restrict ourselves to looking only at the algebraic integers of modulus $\ge 1$ and estimating the probability.\\
\textbf{Case I.} Suppose $|z|>1+\cfrac{u_N}{N}$ where $u_N\geq  (\log N)^2$ and  $P_{f,N}(z)=0.$
Using Proposition 5.4 from \cite{dusart} we can pick primes $p_1,\ldots, p_r$ in the corresponding intervals 
$\left[\cfrac{N}{2}+2i\delta, \cfrac{N}{2}+(2i+1)\delta\right]$ for $0\leq i \leq r-1$
where $r\sim \frac{1}{100}(\log N)^2$, $\delta =10\cfrac{N}{(\log{N})^2}$ and $N$ is large enough.\\
Suppose that for two distinct realizations of random multiplicative functions $f_1\neq f_2$ we have that their values agree on the set $\{1,\ldots, N\}\setminus \{p_1,p_2,\ldots, p_r\}$. Moreover, assume that $z$ is a root of $P_{f_1,N}(X).$ We claim that $z$ can not be a root of $P_{f_2,N}(X)$.
Suppose to the contrary that $P_{f_1,N}(z)=P_{f_2,N}(z)$ and let $L$ be the largest index such that $f_1(p_L)\neq f_2(p_L)$. Note that 
$$\sum_{k=1}^{r} (f_1(p_j)-f_2(p_j)) z^{p_k-1}=\sum_{k=1}^{L} (f_1(p_j)-f_2(p_j) )z^{p_k} =0.$$
By the triangle inequality we have that 
$$0=\left|\sum_{k=1}^{L} (f_1(p_j)-f_2(p_j)) z^{p_k} \right|\geq 2\left(|z|^{p_L}-\sum_{k=1}^{L-1} |z|^{p_k}\right)$$ and thus
$$|z|^{p_L}\leq \sum_{k=1}^{L-1} |z|^{p_k} \hspace{2mm} \text{or equivalently} \hspace{2mm} \sum_{k=1}^{L-1}\cfrac{1}{|z|^{p_L-p_k}}\geq 1  .$$
 On the other hand, by our choice of the primes we know that for every $1\leq i \leq r-1,$ we have $p_{i+1}-p_i\geq \delta$. Thus
$$\sum_{k=1}^{L-1}\cfrac{1}{|z|^{p_L-p_k}}\leq \sum_{k=1}^{\infty} \frac{1}{|z|^{k\delta}}=\frac{1}{|z|^{\delta}-1}$$
implying that $|z|^{\delta}\leq 2$. This is now a contradiction since $$|z|^{\delta}\geq \left(1+\cfrac{u_N}{N}\right)^{10N/(\log N)^2}>1+10\cfrac{u_N}{(\log N)^2}>11.$$
Since the values of $\{f(p_i)\}_{i\le r}$ are independent, we conclude that 
$$\mathbb{P}(P_{f,N}(z)=0)\leq \frac{1}{2^r}\ll \frac{1}{N}$$
at all such $z$, which finishes the proof in this case.

\bigskip

\textbf{Case II.} We now suppose that  $1 \leq |z|\leq 1+\cfrac{u_N}{N}$ where $u_N\le (\log N)^2$ and $P_{f,N}(z)=0.$
Since $z$ has to be algebraic, by considering the minimal polynomial of $z$ we can assume that all of its roots satisfy the above bound since otherwise we can use estimates for the probability of vanishing from  Case I.\\
Let $g_z$ be the minimal polynomial of $z,$ which is of degree $d\leq B.$ Then we can estimate the Mahler measure by
$$M(g_z)\leq \left(1+\cfrac{u_N}{N}\right)^B<e^{Bu_N/N}<e^{\log ^3 (N)/N}.$$
On the other hand, we know that if $z$ is algebraic of degree  $d\geq 3$ and not a root of unity (see \cite{dobr}) then by Dobrowolski bound
$$M(g_z)\geq 1+c\left (\frac{\log\log d}{\log d}\right)^3.$$
This clearly contradicts the bound from above.
It remains to consider possible factors of degree $\le 2,$ and from the bound $|z|\leq 1+\cfrac{u_N}{N}$ those are precisely the polynomials $X\pm 1, X^2\pm X+1.$\\ 
Consequently, it remains to estimate the probability that a polynomial vanishes at $z=\pm 1$ or at the roots of unity of order $d$ for $3\leq d\leq B$.
To this end, we rewrite our equation as $$\displaystyle \sum_{N/2\leq p\leq N, \hspace{1mm} p\hspace{1mm} \text{prime}} f(p)z^{p-1}+\sum_{n\neq p}f(n) z^n=0$$
 For the root $z=\pm 1,$ we condition on all values $\{f(p)\}_{p\le N/2}$ and note that we are left to estimate the probability that 
$$\displaystyle \sum_{N/2\leq p\leq N, \hspace{1mm} p\hspace{1mm} \text{prime}} f(p)$$ is equal to a fixed value.
 This is one-dimensional random walk, so by the standard concentration inequality of Berry-Essen \cite{KS12} we have that for $\ell= N/(2\log N)$ 
$$\sup_{Z\in \mathbb{R}}\mathbb{P}\left( \left|\ell^{-1/2}\sum_{N/2<p<N}f(p)-Z\right|\leq \ell^{-1} \right)\lesssim \ell^{-1/2},$$
which gives acceptable contribution.\\
Now if $\zeta$ is a root of unity of order $3\leq d\leq B,$ we pick a set $\mathcal{P}_d\in [N/2,N]$ of primes $p$ with $|\mathcal{P}_d|\gg \cfrac{N}{\log(N)^3}$ and $p\equiv 1\pmod{d},$  again by applying the uniform version of the prime number theorem for arithmetic progressions. By conditioning on all values of $\{f(p)\}_{p\notin\mathcal{P}_d},$ we have to again estimate the probability that$\displaystyle \sum_{p\in\mathcal{P}_d} f(p)$ is equal to a given value and the previous bounds apply with $l=N/\log(N)^.3$
Thus we obtain that the probability 
$$ \mathbb{P}(P_{f,N}(\zeta)=0)\ll N^{-1/2+\varepsilon}$$
for any $\varepsilon>0$ and $\zeta$ a root of unity of order $1\leq d\leq B.$
Combining the bounds in Cases I and II we obtain the desired bound in the lemma.
\end{proof}
The conclusion of Theorem \ref{mainthm} now immediately follows by combining Lemmas \ref{coeffdistr}, \ref{iredfactor}, \ref{rootbound}.

\subsection{Irreducibility of Turyn polynomials}
As mentioned in the introduction, one motivating example for considering polynomials with random multiplicative coefficients is the study of the irreducibility of the Turyn polynomials given by
\[F_{d,p}(z)=\sum_{a=1}^d\left(\frac{a}{p}\right)z^a,\]
see Problem 1 in \cite{varju}.
These polynomials are of significant interest in number theory and analysis, and they play an important role in extremal problems for Littlewood polynomials; see \cite{BKS}, \cite{jedwab}, \cite{mossinghoff-mahler} for some results and for extensive lists of references. Varjú and Xu \cite{varju} proved that, under the assumption of ERH, if a prime $p$ is drawn uniformly at random from the set $[d,h(d)]$ with $h(d)\ge 2^{\pi(d)}d^4$, then
$$\lim_{p\rightarrow \infty}\mathbb{P}[F_{d,p}(X) \hspace{2mm} \text{is irreducible}]=1.$$

Since the distribution of the sequence $\left(\frac{a}{p}\right)$ for $a\le d$ and $p\ge C^d$ is well modeled by random multiplicative functions, one can directly apply our Theorem \ref{mainthm} together with unconditional variants of the character sum bounds from \cite{varju} to obtain the following unconditional corollary.
\begin{cor}
Let $d=2^k.$ Then, if a prime $p$ is drawn uniformly at random from the set $[d,h(d)]$ with $h(d)\ge 2^{2^{\pi(d)}}$, we have
$$\lim_{k\rightarrow \infty}\mathbb{P}[F_{d,p}(X) \hspace{2mm} \text{is irreducible}]=1.$$
\end{cor}

\end{document}